\documentclass[a4paper,12pt]{amsart}

\usepackage{etex}
\usepackage{graphicx}
\usepackage{pictex}
\usepackage{epstopdf}
\usepackage{a4wide}

\usepackage[utf8]{inputenc}
\usepackage{amsmath}
\usepackage{amssymb}
\usepackage{amscd}
\usepackage{amsthm}
\usepackage{proof}
\usepackage{color}
\usepackage{enumerate}
\usepackage{hyperref}

\newtheorem{theorem}{Theorem}[section]
\newtheorem{coro}[theorem]{Corollary}
\newtheorem{lemma}[theorem]{Lemma}
\newtheorem{prop}[theorem]{Proposition}

\newtheorem{definition}[theorem]{Definition}

\theoremstyle{definition}
\newtheorem{Remark}[theorem]{Remark}

\newcommand{\dalpha}{d_\alpha}

\newcommand{\calpha}{c_\alpha}
\newcommand{\ualpha}{u_\alpha}
\newcommand{\valpha}{v_\alpha}

\newcommand{\ealpha}{e_\alpha}

\newcommand{\length}[1]{\lvert #1 \rvert} 
\newcommand{\block}[1]{p^{(#1)}} 

\newcommand{\NN}{\mathbb{N}}
\newcommand{\ZZ}{\mathbb{Z}}
\newcommand{\CC}{\mathbb{C}}
\newcommand{\RR}{\mathbb{R}}

\newcommand{\OOmega}{(\Omega,T)}

\newcommand{\SL}{SL(2,\RR)}

 \DeclareMathOperator{\Aut}{Aut}
 
\DeclareMathOperator{\Ker}{Ker}

\title[Uniformity of locally constant cocycles]{Subshifts with leading sequences, uniformity of cocycles and spectra of Schreier graphs}

\author{Rostislav Grigorchuk}
\address{Mathematics Department, Texas A\&M University, College Station, TX 77843-3368, USA}
\email{grigorch@math.tamu.edu}

\author{Daniel Lenz}
\address{Mathematisches Institut \\Friedrich Schiller
Universit{\"a}t Jena \\07743 Jena, Germany }
\email{daniel.lenz@uni-jena.de}

\author{Tatiana Nagnibeda}
\address{Section de Math\'{e}matiques, Universit\'{e} de Gen\`{e}ve, 2-4, Rue du
Li\`{e}vre, Case Postale 64 1211 Gen\`{e}ve 4, Suisse}
\email{Tatiana.Smirnova-Nagnibeda@unige.ch}

\author{Daniel Sell}
\address{Mathematisches Institut \\Friedrich Schiller
Universit{\"a}t Jena \\07743 Jena, Germany }
\email{daniel.sell@uni-jena.de}

\begin{document}

\begin{abstract}
We introduce a class of subshifts governed by finitely many two-sided infinite words. We call these words leading sequences. We show that any locally constant cocycle over such a subshift is uniform. From this we obtain Cantor spectrum of Lebesgue measure zero for  associated Jacobi operators if the subshift is aperiodic. Our class covers all simple Toeplitz subshifts as well as all Sturmian subshifts. We apply our results to the spectral theory of Schreier graphs for uncountable families of groups acting on rooted trees.
\end{abstract}

\maketitle


\section*{Introduction}
This article is concerned with matrix valued cocycles over
dynamical systems, with applications to spectral theory of Jacobi
operators. More specifically, we consider a uniquely ergodic
dynamical system $\OOmega$, where $\Omega$ is a compact metric
space, $T:\Omega\longrightarrow \Omega$ is a homeomorphism and there
is only one $T$-invariant probability measure on $\Omega$. It is
well-known that such systems admit a uniform ergodic theorem for
continuous functions, see e.g.  \cite{Wal2}.  Hence, for
any continuous $f :\Omega\longrightarrow \CC$, there exists $\lambda\in \CC$ with
$$\lambda = \lim_{n\to\infty} \frac{1}{n}\sum_{j=1}^n f(T^j \omega) \, ,$$
uniformly in $\omega\in\Omega$. A natural question in this context
is whether a similar result holds for matrix-valued functions if
summation is replaced by multiplication and averaging is done by a
logarithmic mean. To be more precise, to any continuous function $A: \Omega\longrightarrow
\SL$ let us associate the \emph{cocycle}
$$A(\cdot,\cdot):  \ZZ   \times\Omega \longrightarrow \SL$$
defined by
\begin{equation*}
A(n,\omega) := \left\{\begin{array}{r@{\quad:\quad}l}
 A(T^{n-1} \omega)\cdots A(\omega)  & n>0\\
 Id & n=0\\
A^{-1} (T^n \omega) \cdots A^{-1} (T^{-1}\omega)  & n < 0.
\end{array}\right.
\end{equation*}
Following \cite{Fur} (cf. \cite{Wal2} as well), we call the
continuous function $A: \Omega \longrightarrow \SL$ \emph{uniform} if the limit
$$\Lambda = \lim_{n \to \infty} \frac{1}{n} \log \| A(n,\omega)\|$$
exists for all $\omega \in \Omega$ and the convergence is uniform on
$\Omega$. If $A$ is uniform, we also call the associated cocycle
$A(\cdot,\cdot)$ uniform. With these definitions, the question about
existence of ergodic averages uniformly in $\omega$ becomes the
question about uniformity of all continuous $\SL$-valued functions,
or equivalently, of their associated cocycles. If $(\Omega,T)$ is
\emph{minimal} (i.e.  $\{T^n\omega: n\in\ZZ\}$ is dense in $\Omega$
for each $\omega\in\Omega$) the definition can be simplified. Then,
uniform convergence already follows if the limit exists for all
$\omega\in \Omega$. This observation has been shared with us by
Benjamin Weiss in a private communication, and we present its proof
at the end of Section \ref{s:background}.

Existence or non-existence of uniform cocycles  has been studied by
various people, e.g. in \cite{Wal2,Her1,Fur, Len3}. In fact, Walters
raises in \cite{Wal2} the question whether  every uniquely ergodic
dynamical system with non-atomic invariant measure $\mu$ admits a
non-uniform cocycle. Walters presents a class of examples admitting
non-uniform cocycles based on results of Veech \cite{Vee}. He also
discusses a further class of examples, namely suitable irrational
rotations, for which non-uniformity was shown by Herman \cite{Her1}.
Furman carries out in \cite{Fur} a careful study of uniformity of cocycles. For strictly ergodic dynamical systems, he characterizes
uniform cocycles with positive $\Lambda$ in terms of uniform
diagonalizability, which in turn is equivalent to uniform
hyperbolicity of the cocycle. In this way the question of existence
of non-uniform cocycles becomes the question of  existence of
hyperbolic but not uniformly hyperbolic cocycles.

A particular class of dynamical system of interest are subshifts.
For them, there exists a finite set $\mathcal{A}$ such that $\Omega$
is a closed subset of $\mathcal{A}^{\ZZ}$ (equipped with product
topology),  which is invariant under the shift $$T :
\mathcal{A}^{\ZZ}\longrightarrow\mathcal{A}^{\ZZ}, (Ts)(n) =
s(n+1).$$ For subshifts, one may first restrict attention to those
$A:\Omega\longrightarrow \SL$ which are locally constant (i.e. admit
an  $N>0$ such that $A(\omega)$ only depends on $\omega(-N)\ldots
\omega (N)$). We will then also call  the associated cocycle
locally constant.

Uniformity of locally constant cocycles is relevant for applications
to the spectral theory of self-adjoint operators arising in the
study of aperiodic order. Such operators have attracted considerable
attention in recent decades, see e.g.\ the surveys
\cite{Sut,Dam,DEG} for more information on such operators,  and
\cite{TAO,BMoo,KLS} for background on aperiodic order. In
particular, it was shown in \cite{Len}, that uniformity of locally
constant cocycles implies Cantor spectrum of Lebesgue measure zero
for discrete Schroedinger operators associated with aperiodic
subshifts. As a consequence, all aperiodic linearly repetitive
subshifts yield Cantor spectrum of Lebesgue measure zero. In
\cite{DL} this was generalized to the substantially larger class of
aperiodic subshifts  satisfying the so-called Boshernitzan
condition. These results do not only hold for discrete Schrödinger
operators but also for the larger class of Jacobi operators
\cite{BP}. By now it seems fair to say that establishing validity of
Boshernitzan condition has become the standard method of proving
Cantor spectrum in this context. Quite remarkably, \cite{LiuQu} gave
results showing Cantor spectrum for Schrödinger operators
associated to certain subshifts, which do not satisfy Boshernitzan
condition. These subshifts are \emph{simple Toeplitz} subshifts. The
methods used in \cite{LiuQu} are involved and tailored to deal with
discrete Schrödinger operators. In particular, \cite{LiuQu} leaves
open the question whether uniformity of all locally constant
cocycles actually holds for their subshifts. In fact, it is not even
clear that the methods of \cite{LiuQu} carry over to Jacobi
operators and their associated locally constant cocycles.

A main result of the present article, Theorem
\ref{thm-simple-toeplitz-lsc},  shows uniformity of all locally
constant cocycles over simple Toeplitz subshifts. In particular, we
provide  the first examples of subshifts without Boshernitzan
condition which satisfy uniformity of \emph{all} locally constant
cocycles. As a consequence we obtain Cantor spectrum of Lebesgue
measure zero for aperiodic  Jacobi operators associated to simple
Toeplitz subshifts.

The application to Jacobi operators  is particularly relevant as it
allows us to apply our results to a problem from spectral graph
theory. Indeed, as has recently been observed in \cite{GLN},
Laplacians on certain infinite graphs arising naturally in the
theory of groups of intermediate growth can be realized as Jacobi
operators on subshifts. In \cite{GLN}, the important example of the
so-called ``first Grigorchuk's group\footnote{This is how the group
is generally known and how we refer to it, in spite of the first
author's reluctance.}", introduced in \cite{Gr80}, was studied from
this viewpoint. The corresponding subshift turned out to be defined
by a primitive substitution and thus linearly repetitive. Hence the
Cantor spectrum of Lebesgue measure zero holds for the anisotropic
Laplacians on the infinite Schreier graphs of this group.

In fact, this group belongs to an uncountable family of groups
\cite{Gr84} that are non-isomorphic and even non-quasi-isometric,
but nevertheless share many properties, in particular, they all act
by automorphisms on the infinite binary tree and are of intermediate
growth, i.e., have word growth strictly between polynomial and
exponential. It is then natural to ask (see \cite{GLN_Survey,GLN})
whether the spectral result from \cite{GLN} holds for other groups
in the family, and even to members of other related uncountable
families of groups, so called spinal groups. Our second main result
(see Corollary~7.2 and Section~7 in general) shows that this is
indeed the case. The part of the proof that realizes the Laplacians
on Schreier graphs as Jacobi operators on subshifts carries over
from \cite{GLN}. However, the subshifts that arise will no longer be
linear repetitive, and in general will not even satisfy the
Boshernitzan condition. However, we show that they are all simple
Toeplitz subshifts, and therefore our main result applied and yields
 Cantor spectrum of Lebesgue measure zero in the aperiodic
(anisotropic) case.

A few words on methods may be in order. The basic task in proving
uniformity of cocycles is to provide lower bounds on the growth of
products of matrices. For locally constant $A$ this amounts to
providing lower bounds on growth along \emph{all finite words}.
Inspired by the considerations for Sturmian subshifts in
\cite{Len4}, we introduce here  a new method to achieve this. This
method may be of interest in itself (beyond the applications given
in the main two results). It  relies on proving growth along
\emph{finitely many infinite words}, which control the whole
subshift in a meaningful way. We call these infinite words
\emph{leading sequences} and call the subshifts admitting them
\emph{subshifts satisfying the leading sequence condition} (LSC). On the
conceptual level, putting forward this class of subshifts is a key
insight of the article. The corresponding technical result claims uniformity of locally constant cocycles over (LSC) subshifts, see Theorem~\ref{l:main-technical}. We show that this class contains all
simple Toeplitz subshifts as well as all Sturmian subshifts.


The article is organized as follows. In Section~\ref{s:background} we recall some basic notation and results concerning subshifts and cocycles. In Section~\ref{s:subshifts} we introduce the leading sequence condition (LSC) for subshifts. We discuss general properties of (LSC) subshifts and prove the main technical result: uniformity of locally constant cocycles on (LSC) subshifts (Theorem~\ref{l:main-technical}). In Section~\ref{s:spectral} we derive Cantor spectrum for Jacobi operators associated to aperiodic (LSC) subshifts. In Section~\ref{s:combinatorial} we provide combinatorial conditions ensuring (LSC). Equipped with these combinatorial conditions, we then show that simple Toeplitz subshifts and Sturmian subshifts satisfy condition (LSC) in Section~\ref{s:simple} and Section~\ref{s:sturmian} respectively. Finally, Section~\ref{s:family} contains the application of our results to the study of spectra of Laplacians on infinite Schreier graphs of Grigorchuk groups and spinal groups acting on the infinite binary tree. The material of the Sections~\ref{s:subshifts} to \ref{s:sturmian} constitutes a part of the PhD thesis of one of the authors \cite{Sell-PhD}.

\medskip

\bigskip

{\small \textbf{Acknowledgements.} R.~G. and D.~L. gratefully acknowledge hospitality of the Department of Mathematics of the University of Geneva on various occasions during the last five years. The work of R.~G. is supported by the Simons Foundation through Collaboration Grant 527814. R.~G. and T.~N. gratefully acknowledge support of the Swiss National Science Foundation and of the grant of the Government of the Russian Federation No 14.W03.31.0030. The work of D.~S. is supported by a PhD scholarship from \textit{Landesgraduiertenstipendium - Th\"uringen}. D.~S. would also like to express his gratitude for an invitation to the Department of Mathematics of the University of Geneva in 2017. }

\section{Background on subshifts and uniform cocycles}\label{s:background}
In this section we recall some background on our main actors. These
are subshifts over a finite alphabet and uniform cocycles.

\bigskip

Let $\mathcal{A}$ be a finite set called the \emph{alphabet}. The
elements of the free monoid $\mathcal{A}^* =\cup_{n=0}^\infty
\mathcal{A}^n$ are denoted as \emph{finite words} over
$\mathcal{A}$, where $\mathcal{A}^0 = \{\epsilon\}$ with the
\emph{empty word} $\epsilon$. We  will freely use standard
notation concerning words (see e.g. \cite{Lot}). In particular, we define the \emph{length}
of a word $v$ by $|v| = n$ if $v\in\mathcal{A}^n$ for some
$n\in\NN$. Moreover, we let the concatenation $x y$ of $x =
x(1)\ldots x(N)$ and $y=y(1)\ldots y(M)$
  be given by $x y =x(1)\ldots x(N) y(1)\ldots y(M)$ and call
  $x$  then a \emph{prefix} of $xy$.

A pair $(\Omega,T)$ is called a \emph{subshift} over
$\mathcal{A}$, if $\Omega$ is a closed subset of $\mathcal{A}^{\ZZ}$
(with product topology) and invariant under the shift $T:
\mathcal{A}^{\ZZ}\longrightarrow \mathcal{A}^{\ZZ}$, $(T s)(n) :=
s(n+1)$.  Whenever $(\Omega,T)$ is a subshift over $\mathcal{A}$ and $v$ is a
word of length $n$ over $\mathcal{A}$, we say that $v$ \emph{occurs} in $\omega\in\Omega$
at $k\in\ZZ$ if
$$v = \omega (k)\ldots \omega (k+n-1).$$ We denote the set of all
words of length $n$ occurring in $\omega\in \Omega$  by
$\mathcal{W}(\omega)_n$ and define the set of all finite words associated to $(\Omega,T)$ by
$$\mathcal{W}(\Omega):=\bigcup_{n\in \NN,\omega\in\Omega} \mathcal{W}(\omega)_n.$$

As is well known, both minimality and unique ergodicity can be
characterized via $\mathcal{W}(\Omega)$ (see for example Section~5.1
in \cite{PFogg}). More specifically, the subshift is minimal if
$\mathcal{W}(\Omega)_n =\mathcal{W}(\omega)_n$ for any
$\omega\in\Omega$ for any $n\in\NN$. A subshift is uniquely ergodic
if the limit
$$\lim_{n\to\infty} \frac{\sharp_v \omega(1)\ldots\omega(n)}{n}=:\nu(v)$$
exists uniformly in $\omega\in \Omega$ for any
$v\in\mathcal{W}(\Omega)$. Here, $\sharp_v \omega (1)\ldots\omega
(n)$ denotes the number of occurrences of $v$ in $\omega(1)\ldots
\omega (n)$ i.e. the cardinality of the set
$$\{k\in\{1,\ldots, n-|v|+1\} : \omega(k)
\ldots \omega(k +|v|-1) = v\}.$$ The term $\nu(v)$ is then called
the \emph{frequency} of $v$ and uniformity of the limit above is
referred to as \emph{uniform existence of frequencies}.

For a finite word $w= w(1)\ldots w(n)$ we define the reflected word
$w^R$ via $w^R = w(n) \ldots w(1)$. Similarly, for a one-sided
infinite word $\eta : \NN\longrightarrow \mathcal{A}$ we denote by
$\eta^R$ the one-sided infinite word arising by reflecting $\eta$
(i.e. $\eta^R :-\NN \longrightarrow \mathcal{A}, \eta^R (n)=\eta
(-n)$).  We write $|$ to denote the position of the origin. More
specifically for an  $\omega : \ZZ\longrightarrow \mathcal{A}$ we
write
$$\omega =  \varrho|\eta $$
whenever $\eta : \NN\cup\{0\} \longrightarrow \mathcal{A}$ agrees
with  the restriction of $\omega$ to $\NN\cup \{0\}$ and $\varrho :
-\NN\longrightarrow \mathcal{A}$ agrees with the restriction of
$\omega$ to $-\NN$. A function $f$ from a subshift is called
\emph{locally constant} if there exists an $N\in \NN$ such that
\begin{equation*}
f(\omega)= f(\varrho), \;\:\mbox{whenever}\:\; \omega(-N) \ldots
\omega(N) = \varrho(-N) \ldots \varrho(N).
\end{equation*}
Clearly, any locally constant function into a topological space is
continuous. As mentioned already in the introduction our main
concern are locally constant functions $A : \Omega\longrightarrow
\SL$. Any such $A$ gives rise to a cocycle and we will also call
this cocycle locally constant. For our subsequent dealing with
locally constant cocycles we include two simple observations. One
observation shows that we can 'chop off' finite pieces of cocycles
without changing the exponential behaviour, the other is that we can
take inverses without changing norms.

\begin{prop}\label{p:nutzlich} Let $(\Omega,T)$ be a subshift and $A :
\Omega\longrightarrow \SL$ locally constant. Then the following
holds:

\begin{itemize}
\item[(a)] $\|A(-n,\omega)\| = \|A(n,T^{-n} \omega)\|$ for all $n\in \NN$ and
$\omega\in \Omega$.

\item[(b)] There exists for any $ N \in \NN $ a $c(N)>0$ with
\[ \log\|A(n-N,  T^N \omega)\|   - c(N) \leq \log\|A(n,\omega)\| \leq \log\|A(n-N, T^N  \omega)\| + c(N) \]
for all $\omega\in\Omega$ and all $ n \in \ZZ $.
\end{itemize}

\end{prop}
\begin{proof} (a) This is a direct consequence of the definition of $A(\cdot,\cdot)$ and
$\|B\| =
\|B^{-1}\|$ for any $B\in \SL$.

\smallskip

(b) By continuity of $A$ we easily infer
$$c(N):=\sup_{\omega\in\Omega} \{\log \|A(N,\omega)\| \}<\infty.$$ Since $ A( n , \omega ) = A( n-N , T^N \omega ) A( N , \omega ) $, the desired statement follows now from the inequality
$$\frac{1}{\|D^{-1}\|} \|C\| \leq \|C D\|\leq \|D\| \|C\|$$
 for matrices $C,D\in \SL$.
\end{proof}

We finish this section by discussing how the definition of
uniformity of cocycles can be simplified in the case of minimal
dynamical systems. More specifically, in that case uniformity of
convergence follows once convergence  holds for all elements of the
dynamical system.  We learned this result  from Benjamin Weiss. The
result and its proof do not seem to have appeared in print and with
his kind permission we include it here.
\begin{theorem} Let $(\Omega,T)$ be a minimal dynamical system and $f_n :\Omega\to \RR$ be continuous functions for $n\in \NN$ with
$$ f_{n+m} (\omega) \leq f_n (\omega) + f_m (T^n \omega)$$
for all $\omega\in \Omega$ and $n,m\in \NN$. If $f_n (\omega)/n \to
\phi (\omega)$ for all $\omega\in \Omega$ then $\phi$ is constant
and the convergence is uniform.
\end{theorem}

The \textit{proof} is split in two parts:
\begin{itemize}
\item[(I)] Let $a_n:=\min\{ f_n (\omega) / n: \omega\in \Omega\}$ and $A:=\liminf a_n$.  Then, $\phi (\omega)\leq A$ for a dense $G_\delta$ set of $\omega$'s.
\item[(II)]
Let $b_n :=\max\{ f_n (\omega) / n: \omega\in \Omega\}$ and
$B:=\limsup b_n$.  Then, $\phi (\omega)\geq B$ for a dense
$G_\delta$ set of $\omega$'s.
\end{itemize}
It is clear that the theorem follows from $(I)$ and $(II)$. It
remains to show $(I)$ and $(II)$.

\textit{Proof of $(I)$:}  Fix $\alpha >A$ and consider
$$ E_\alpha^n :=\{\omega : \inf_{k\geq n} f_k (\omega)/ k < \alpha\}.$$
Then, $E_\alpha^n$ is open for any $n\in \NN$. It remains to show
that each $E_\alpha^n$  is dense. (Then, the statement will follow
from the Baire category theorem by taking intersections over $n$ and
suitable sequences $\alpha$ converging to $A$.) To show denseness of
$E_\alpha^n$ we let $W$ be an open set in $X$.  By minimality there
exists an $N$ such that $\cup_{j=0}^N T^j W = X$. We now choose $m$
sufficiently large with $a_m < \alpha$. Then,
$$V:=\{\omega\in\Omega : f_m (\omega)/m < \frac{a_m + \alpha}{2}\}$$
is open and non-empty (by definition of $a_m$). Thus, we can find
$x\in W$ and $j\in \{ 0,\ldots, N \}$ with  $T^j x \in V$. By
subadditivity we  then have
$$ \frac{f_{m+j} (x)}{m + j} \leq \frac{f_{j} (x)}{j}\frac{j}{m+ j} + \frac{f_m (T^j x)}{m} \frac{ m }{m+j}< \alpha.$$
Here, the last inequality follows from the definition of $V$ and the
largeness of $m$. This, shows $x\in E_\alpha^n$. As $W$ was
arbitrary the desired denseness statement holds.

\medskip

\textit{Proof of $(II)$:}  The condition on the $f_n$ easily yields
$$ f_{m-n} (\omega)   \geq - f_n (T^{-n} \omega) + f_m (T^{-n} \omega)$$
for all $m,n$ with $n \leq m$. We can now consider $\beta < B$,
$$ F_\beta^n :=\{\omega : \sup_{k\geq n} f_k (\omega)/k > \beta\}$$
and use $\cup_{j=0}^N T^{-j} W = X$  and
$$ U:=\{\omega : f_m (\omega)/m > \frac{b_m + \beta}{2}\}$$
to mimick the proof of $(I)$.

\medskip
This finishes the proof of the theorem. \hfill \qed

\medskip

As a consequence of the previous theorem a cocycle $A$ over a
minimal dynamical system $(\Omega,T)$ is uniform if and only if
$$\lim_{n\to \infty}\frac{1}{n} \log\|A(n,\omega)\|$$ exists for all
$\omega\in\Omega$.

\section{Subshifts satisfying (LSC) and the main (technical) result}\label{s:subshifts}
In this section we introduce the class of subshifts that is the main
concern in this article, discuss  some of their basic properties and
state our main technical result for these subshifts.

\bigskip

Consider a subshift $(\Omega,T)$.  The subshift is said to satisfy
the \emph{combinatorial leading sequence condition}   if there
exists a natural number $r\in \NN$  and finitely many $\omega^{(j)} \in\Omega$,
$j=1,\ldots, r$, such that the following holds:
\begin{itemize}
\item[($\alpha$)] There exists $N\in\NN$   with
$$\mathcal{W}(\Omega)_n  = \bigcup_{j=1}^r  \{\omega^{(j)} (-k+1 )\ldots
\omega^{(j)}(-k + n) : k = 0, \ldots, n\}$$
 for all $n\in \NN$ with $n\geq N$.
\end{itemize}
In this case the words $\omega^{(j)}$, $j=1,\ldots, r$, are called
the \emph{leading words} or \emph{leading sequences} of the
subshift. Clearly, $(\alpha)$ is a condition concerning
combinatorics on words.  A subshift satisfying the combinatorial
leading sequence condition with leading words $\omega^{(j)}$,
$j=1,\ldots, r$, is said to satisfy the \emph{cocycle leading
sequence condition} if, for every locally constant function $A : \Omega\longrightarrow \SL$, the following two statements holds:
\begin{itemize}
\item[($\beta$)]{For every $j\in \{1,\ldots, r\}$ the limits
$$\lim_{n\to \pm \infty} \frac{\log\|A(n,\omega^{(j)})\|}{|n|} $$
exist. Moreover, all limits have the same value.}

\item[($\gamma$)]{For every $j\in \{1,\ldots,r\}$ and every $v\in \RR^{2} \setminus \{ 0 \} $, at most one of the limits $ \lim_{n \to \pm \infty} \frac{1}{\lvert n \rvert}\log\|A(n,\omega^{(j)})v\| $ is negative.}
\end{itemize}
The conditions $(\beta), (\gamma)$ involve cocycles and are -
apriori - hard to check. In  Section \ref{s:combinatorial} we will
provide sufficient combinatorial conditions for validity of
$(\beta)$ and $(\gamma)$. These sufficient conditions will be shown
to hold in the case of simple Toeplitz subshifts and Sturmian
subshifts in Section \ref{s:simple} and Section \ref{s:sturmian}.

\begin{definition}[Leading sequence condition (LSC)]
A subshift satisfying $(\alpha),(\beta)$ and $(\gamma)$ is said to
satisfy the \emph{leading sequence condition} (LSC).
\end{definition}

\bigskip

We next gather some simple properties of subshifts satisfying (LSC).

\begin{prop} Let $(\Omega,T)$ be a subshift satisfying (LSC) with leading words $\omega^{(j)}$, $j=1,\ldots, r$. Then,
the following holds:\\
\\ (a) The subshift is uniquely ergodic. \\
\\ (b) The inequality $\sharp \mathcal{W}(\Omega)_n\leq r (n+1) $ holds for
all $n\in\NN$ larger than a suitable $N\in\NN$.
\end{prop}

\begin{proof} (a) It suffices to show uniform existence of
frequencies of words. Let $v\in \mathcal{W}(\Omega)$ be arbitrary.
Define $A : \Omega\longrightarrow \SL$ by $A(\omega) = I$ if
$\omega(0)\ldots \omega (|v|-1) \neq v$ and
\begin{equation*}  A(\omega) =\left(
\begin{array}{cc} 2  & 0\\ 0 & 1/2
\end{array}   \right)
\end{equation*}
otherwise. Clearly, $A$ is locally constant. Now, it it not hard to
see that $$ \frac{\log\|A(n,\omega^{(j)})\|}{n} = \log 2
\frac{\sharp_v \omega^{(j)} (0) \ldots \omega^{(j)}(|v|+n-2) }{n}$$ for
$n > 0$ and, similarly, $$ \frac{\log\|A(n,\omega^{(j)})\|}{|n|} =
\log 2 \frac{\sharp_v \omega^{(j)} (n) \ldots \omega^{(j)}(|v|-2)
}{|n|}$$ for $n < 0$. So, from ($\beta$) we infer that the
frequencies of words exist along all half-sided sequences of the
form $\omega^{(j)}(1)\omega^{(j)}(2)\ldots$ and  $\ldots
\omega^{(j)}(-1)\omega^{(j)}(0)$, $j=1,\ldots, r$. From ($\alpha$)
we then easily  obtain uniform existence of frequencies of words.

\smallskip

(b)  This is immediate from ($\alpha$).
\end{proof}

\begin{Remark} While  we do not need it, we note that (LSC) is
stable under morphism. More specifically, consider a subshift
$(\Omega,T)$ over $\mathcal{A}$ satisfying (LSC). Let $\mathcal{B}$
be a finite set and  $\varphi : \mathcal{A} \longrightarrow
\mathcal{B}^*$ be arbitrary.  Define  $\Omega_\varphi\subset
\mathcal{B}^\ZZ$ as the set of all translates of sequences of the
form
$$....\varphi (\omega(-1)) \varphi (\omega(0)) \varphi
(\omega(1))\ldots ...$$ for $\omega\in\Omega$. Then,
$(\Omega_\varphi,T)$ is a subshift satisfying (LSC) as well. We
leave the details to the reader.
\end{Remark}

We now turn towards our main result.  We start by recalling a lemma
essentially due to Ruelle \cite{Rue} (see \cite{Len3,LS} as well).

\begin{lemma}\label{l:ruelle}

Let $(A_n)$ be a sequence of matrices in $\SL$ with $ \sup_{n\in
  \NN} \|A_{n+1} A_{n}^{-1}\|<\infty$ and assume that $\Lambda
  =\lim_n \frac{\log\|A_n \ldots A_1\|}{n}$
   exists and is positive. Then, there exists a unique one-dimensional subspace
   $V\subset \RR^2$
  with $$\lim_n\frac{\log\|A_n \ldots A_1 v\|} {n} = -\Lambda \mbox{ and }
\lim_n \frac{\log\|A_n \ldots A_1  u\|} {n} = \Lambda$$ for all
$v\in V$ with $v\neq 0$ and all  $u\notin V$.
\end{lemma}

After this preparation we can come to the main technical result of
the article.

\begin{theorem}[Main technical result]\label{l:main-technical}
Assume that the minimal uniquely ergodic subshift
 $(\Omega,T)$ satisfies (LSC). Then, every
locally constant function $A: \Omega \longrightarrow \SL$  is uniform.
\end{theorem}
\begin{proof} By ($\beta$) there exist a $\Lambda\geq 0$ with
$\Lambda =\lim_{n\to \pm\infty} \frac{\log\|A(n,\omega^{(j)})\|}{|n|}$
for every $j=1,\ldots, r$. If $\Lambda=0$ the desired statement
follows rather easily from ($\alpha$). So, we consider now the case
$\Lambda \neq 0$. By Theorem 3 of \cite{Len3} it suffices to show
that there exists a $\delta>0$ with
$$\frac{\log\|A(n,\omega)\| }{n} \geq \delta$$ for all
$\omega\in\Omega$ and all sufficiently large $n$. This follows from
($\beta$) and ($\gamma$). Here are the details: Assume without loss
of generality that $A(\omega)$ only depends on $\omega(0)$. (The
general case can be treated by Proposition \ref{p:nutzlich}.) By Lemma~\ref{l:ruelle} there exists for each $ j \in \{1,\ldots, r \} $ a one dimensional subspace $ V^{(j)}_{+} \subset \RR^{2} $ with
\[ \frac{\log \lVert A( n , \omega^{(j)} ) v \rVert }{n} \to -\Lambda , n \to \infty , \]
whenever $ v \in V^{(j)}_{+} \setminus \{ 0 \} $ and a one dimensional subspace $ V^{(j)}_{-} \subset \RR^{2} $ with
\[ \frac{ \log \lVert A( -n , \omega^{(j)} ) v \rVert }{n} \to -\Lambda , n \to \infty , \]
whenever $ v \in V^{(j)}_{-} \setminus \{ 0 \}$. Now, by $(\gamma)$ we infer that $ V^{(j)}_{+} \neq V^{(j)}_{-} $ for each $ j \in \{1 , \ldots , r \} $. So, again, by Lemma~\ref{l:ruelle} we have
\[ \frac{\log \lVert A( n , \omega^{(j)} ) v \rVert }{n} \to \Lambda , n \to \infty \]
whenever $ v \in V^{(j)}_{-} \setminus \{ 0 \} $. Now fix $ \hat{v} \in  V^{(j)}_{-} \setminus \{ 0 \} $. Then there exists $ n_{1} \in \NN $ with
\[ \frac{ \log \lVert A( -n , \omega^{(j)} ) \hat{v} \rVert }{n} \leq - \frac{\Lambda}{2} \quad \text{and} \quad \frac{\log \lVert A( n , \omega^{(j)} ) \hat{v} \rVert }{n} \geq \frac{\Lambda}{2} \]
for all $ n \geq n_{1}$. Moreover, there exists a number $n_{2}$ with
\[ \frac{1}{n} \cdot \lvert \; \max_{l : \, \lvert l \rvert < n_{1}} \left( \log \lVert A( l , \omega^{(j)} ) \hat{v} \rVert \right) \; \rvert  \leq \frac{\Lambda}{8} \]
for all $n \geq n_{2}$. Let now $ n_{-}, n_{+} \geq 0 $ with $ N := n_{-} + n_{+} \geq \max \{ 2 n_{1} , n_{2} \} $. It is easy to see that
\[ A( N , T^{-n_{-}} \omega^{(j)} ) = A( n_{+} , \omega^{(j)} ) \cdot A( -n_{-} , \omega^{(j)} )^{-1} \, .\]
We denote $ \hat{u} := A( -n_{-} , \omega^{(j)} ) \hat{v} $ and obtain
\begin{align*}
\log \lVert A( N , T^{-n_{-}} \omega^{(j)} ) \rVert &\geq \log \left( \frac{\lVert A( N , T^{-n_{-}} \omega^{(j)} ) \hat{u} \rVert }{\lVert \hat{u} \rVert} \right) \\
&= \log \lVert A( n_{+} , \omega^{(j)} ) \hat{v} \rVert - \log \lVert A( -n_{-} , \omega^{(j)} ) \hat{v} \rVert \, .
\end{align*}
Clearly, $n_{-}$ and $n_{+}$ cannot both be smaller than $n_{1}$, since we assumed $ n_{-} + n_{+} \geq 2 n_{1} $. If $ n_{-} , n_{+} \geq n_{1} $ holds, then the definitions of $n_{1}$ and $\hat{v}$ imply
\[ \frac{ \log \lVert A( N , T^{-n_{-}} \omega^{(j)} ) \rVert }{N} \geq \frac{ \log \lVert A( n_{+} , \omega^{(j)} ) \hat{v} \rVert }{n_{+}} \frac{n_{+}}{N} - \frac{ \log \lVert A( -n_{-} , \omega^{(j)} ) \hat{v} \rVert }{n_{-}} \frac{n_{-}}{N} \geq \frac{\Lambda}{2} \, .\]
If $ n_{-} < n_{1} $ and $ n_{+} \geq n_{1} $ hold, we can use $ \frac{n_{+}}{N} \geq \frac{1}{2} $ as well as $ n_{-} < n_{1} $ and $ N \geq n_{2} $ to obtain
\[ \frac{ \log \lVert A( N , T^{-n_{-}} \omega^{(j)} ) \rVert }{N} \geq \frac{ \log \lVert A( n_{+} , \omega^{(j)} ) \hat{v} \rVert }{n_{+}} \frac{n_{+}}{N} - \frac{ \log \lVert A( -n_{-} , \omega^{(j)} ) \hat{v} \rVert }{N} > \frac{\Lambda}{8} \, . \]
The remaining case ($ n_{-} \geq n_{1} $ and $ n_{+} < n_{1} $) can be treated similarly. By Proposition~\ref{p:nutzlich}, we infer that there exists a $ \delta > 0 $ with
\[ \frac{ \log \lVert A( N , T^{-n_{-}+1} \omega^{(j)} ) \rVert }{N} \geq \delta \]
Now, by ($\alpha$), for every $\omega$ and every sufficiently large $n$, there exists $ j \in \{ 1 , \hdots , r \} $ and $ n_{-} ,n_{+} \geq 0 $ such that
\[ \omega( 0 ) \hdots \omega( n-1 ) = \omega^{(j)}( -n_{-}+1 ) \hdots \omega^{(j)}( 0 ) \omega^{(j)}( 1 ) \hdots \omega^{(j)}( n_{+} ) \, . \]
Here, for $ n_{-} = 0 $, the part $ \omega^{(j)}( -n_{-} ) \hdots \omega^{(j)}(-1)$ denotes the empty word and similarly for $ n_{+} = 0 $.
We obtain $ A( n , \omega ) = A( n_{-} + n_{+} , T^{-n_{-}+1} \omega^{(j)} ) $ and the desired statement follows.
\end{proof}

\begin{Remark} The definition of (LSC) may be weakened and still
allow for the above result to hold. Details are discussed in this
remark:

(a)  Invoking the avalanche principle of Bourgain / Jitomirskaya
\cite{BJ} as in in \cite{DL} it is not hard to see that the previous
result remains valid if condition $(\alpha)$ is weakened to
$(\widetilde{\alpha})$:  There exists an $r\in \NN$,  finitely many
$\omega^{(j)} \in\Omega$, $j=1,\ldots, r$, as well as a sequence
$(l_n)$ of natural numbers with $l_n\to \infty$,  such that the
following holds:
\begin{itemize}
\item[$(\widetilde{\alpha})$] There exists $N\in\NN$   with
$$\mathcal{W}(\Omega)_{l_n}  = \bigcup_{j=1}^r  \{\omega^{(j)} (-k+1 )\ldots
\omega^{(j)}(-k + l_n) : k = 0, \ldots, l_n\}$$
 for all $n\in \NN$ with $n\geq N$.
\end{itemize}
We refrain from giving details as our main examples satisfy
condition $(\alpha)$.

(b) Note that one could also allow for infinitely many leading
sequences provided the convergence in $(\beta)$ is uniform over the
family of leading sequences.
\end{Remark}

\section{Spectral theory of Jacobi operators associated to (LSC)
subshifts}\label{s:spectral} In this section we introduce the Jacobi
operators associated to a subshift and then present our first main result, which is a spectral
consequence of Theorem~\ref{l:main-technical}.

\bigskip

Consider a dynamical system  $(\Omega,T)$. To continuous functions
 $f: \Omega \longrightarrow \RR \setminus \{ 0 \}$,  $g : \Omega \longrightarrow \RR$ we associate a family of
\emph{discrete operators} $(H_\omega)_{\omega\in \varOmega}$.
Specifically, for each $\omega \in\varOmega$,  $H_\omega$ is a
bounded selfadjoint operator from $\ell^2 (\ZZ)$ to $\ell^2 (\ZZ)$
acting via
$$(H_\omega u ) (n) = f(T^n \omega) u (n-1) +
f(T^{n+1}\omega)  u (n+1) + g(T^n \omega) u (n)$$ for $u \in \ell^2
(\ZZ)$ and $n\in\ZZ$. In the case $f\equiv 1$ these operators are
known  as \emph{discrete Schr\"odinger operators}. For general $f$
the name \emph{Jacobi operators} is often used in the literature.
The spectrum of $H_\omega$ i.e. the set of $E\in \RR$ such
that  $(H_\omega - E)$ is not invertible, is denoted by $\sigma
(H_\omega)$. If $(\Omega,T)$ is minimal, then the spectrum of
$H_\omega$ does not depend on $\omega\in \Omega$. We denote it by
$\Sigma(f,g)$. Thus, we have
$$\Sigma(f,g) = \sigma(H_\omega)$$
for all $\omega\in\Omega$ in the minimal case. The spectral theory
of the $H_\omega$ is intimately linked with behaviour of solutions
$u:\ZZ\longrightarrow \CC$ to the equation
$$\;\: f(T^n \omega) u (n-1) +
f(T^{n+1}\omega)  u (n+1) + g(T^n \omega) u (n) - E u(n) = 0$$ for
$E\in\RR$. The behaviour of such solutions in turn can be captured
via the function
$$M^E : \Omega \longrightarrow \SL,
M^E( \omega ) := \left( \begin{array}{cc} \frac{E- g(T\omega)}{f(T^{2}\omega)} &
\frac{-1}{f(T^{2}\omega)}\\f(T^{2}\omega) & 0
\end{array} \right).
$$
More specifically, as is well known and not hard to see, $u$ is a
solution of the preceding equation if and only if
$$\widetilde{u} (n) = M^E (n,\omega)\widetilde{u} (0)$$
for all $n\in\ZZ$, where
$$\widetilde{u}(n) = \left( \begin{array}{c}  u(n+1) \\  f(T^{n+1} \omega) u(n) \end{array}
\right).$$  We call $M^E$ the \emph{Jacobi cocycle} associated to
the energy $E$. Note that it belongs indeed to $\SL$ (and it is
exactly to ensure this that we introduced the  factor $f(T^{n+1}\omega)$
in various places above.)

Now, uniformity of the cocycles $M^E$  implies Cantor spectrum of
Lebesgue measure zero. In the Schrödinger case this was shown
 in \cite{Len}. This result was then extended to
Jacobi operators in \cite{BP}. To state the result in a precise
manner, we need one more piece of notation. We say that $(f,g) :
\Omega\longrightarrow \RR^2$ is \emph{periodic} if there exists a
natural number $P\geq 1$ with $(f(T^P\omega), g(T^P \omega)) =
(f(\omega), g(\omega))$ for all $\omega\in\Omega$.

\begin{lemma}[Theorem 3 in \cite{BP}]\label{l:cantor-jacobi} Let $(\Omega,T)$ be a
uniquely ergodic minimal  dynamical system and
$f,g:\Omega\longrightarrow\RR$ continuous  with $f(\omega)\neq 0$
for all $\omega\in\Omega$. Assume that $(f,g)$ is not periodic.
Then, $\Sigma(f,g)$ is a Cantor set of Lebesgue measure zero if
$M^E$ is uniform for every $E\in\RR$.
\end{lemma}

Given this result it is not hard to prove the main application of
our article.

\begin{theorem}\label{t:main-abstract} Let $(\Omega,T)$ be
a uniquely ergodic minimal  subshift satisfying (LSC). Let $f,g :
\Omega \longrightarrow \RR$ be continuous functions taking only
finitely many values such that   $f(\omega)\neq 0$ for all
$\omega\in\Omega$ and $(f,g) : \Omega\longrightarrow \RR^2$ is not
periodic. Then, the spectrum $\Sigma(f,g)$ of the associated Jacobi
operators is a Cantor set of Lebesgue measure zero.
\end{theorem}
\begin{proof} As the continuous functions $f,g$ take only finitely many values, they are
locally constant. Thus, the associated cocycle $M^E$ is locally
constant for every $E\in \RR$. Hence, by (LSC) and Theorem
\ref{l:main-technical} all $M^E$ are uniform. Now, the desired
statement follows directly from Lemma \ref{l:cantor-jacobi}.
\end{proof}

As we will see in the subsequent sections, there are ample classes of examples to which this theorem applies.

\section{Combinatorial criteria for (LSC)}\label{s:combinatorial}
In this section we discuss  combinatorial conditions on the subshift
ensuring (LSC). Clearly, ($\alpha$) is already a combinatorial
condition. So, we mainly work on providing combinatorial criteria
for $(\beta)$ and ($\gamma$).   In the end, we provide a
sufficient condition for a subshift to satisfy (LSC). While this
condition may not seem particularly pleasing  at first sight, it
turns out that it can rather easily be checked in examples. In fact,
this is how we will treat the examples discussed in the subsequent
sections  of the article.

\bigskip

Let $\mathcal{A}$ be a finite alphabet and $(\Omega,T)$ a subshift
over $\mathcal{A}$.  A function $F
:\mathcal{W}(\Omega)\longrightarrow \RR$ is called
\emph{subadditive} if $F(xy) \leq F(x) + F(y)$ holds for all
$x,y\in \mathcal{W}(\Omega)$ with $xy\in\mathcal{W}(\Omega)$. For a subadditive function, we define $F^{(n)} :=\max\{ \frac{F(x)}{n} : |x| = n\}$ and $\phi
:\NN\longrightarrow \RR, \phi(n) :=n F^{(n)}$. Then, $\phi (n+m)
\leq \phi (n) + \phi(m)$ for all $n,m\in \NN$ by subadditivity of
$F$. Thus, we infer $$\lim_{n\to \infty} \frac{\phi(n)}{n} = \inf_n
\frac{\phi(n)}{n} = :\overline{F}.$$

The sequence $p : \NN\longrightarrow \mathcal{A}$ is said to satisfy
the condition (PQ) if there exists a $c>0$ such that for any prefix
$v$ of $p$ the inequality
$$\liminf_{n\to \infty} \frac{\widetilde{\sharp}_v p(1)\ldots
p(n)}{n}|v| \geq c$$ holds. Here, $\widetilde{\sharp}_v w$ denotes
the maximal number of mutually disjoint copies of $v$ in $w$.

After these preparations we can now provide a characterization of
those $p$ for which $\lim_n \frac{F(p(1)\ldots p(n))}{n} =
\overline{F}$ holds.  The statement can be seen as a variant of a
main result in \cite{Len2} and the proof is inspired by methods from
\cite{Len2}.

\begin{lemma}[Combinatorial condition for ($\beta$)]
\label{lem:CombCondB}
Let $(\Omega,T)$ be a subshift, $\omega\in \Omega$ and $p = \omega(1 )\omega(2) \ldots ...$.

(a) The limit $\lim_{n\to\infty} \frac{F(p(1)\ldots p(n))}{n}$ exists for a subadditive function $F:\mathcal{W}(\Omega)\longrightarrow\RR$ if the following two assumptions hold:
\begin{itemize}
\item{$p$ satisfies (PQ);}
\item{there exists a sequence $p^{(n)}$ of prefixes of $p$ with $\lim_{n\to \infty} \frac{F(p^{(n)})}{|p^{(n)}|} = \overline{F}$.}
\end{itemize}

(b) If $\lim_{n\to \infty} \frac{F(p(1)\ldots p(n)) }{n} $ exists for every subadditive $F$, then $p$ satisfies (PQ).
\end{lemma}

\begin{Remark} It is possible to replace (PQ) in the previous lemma  by the seemingly
weaker condition (PW) given as follows: there exists a $c>0$ such that for any prefix
$v$ of $p$ the inequality
$$\liminf_{n\to \infty} \frac{{\sharp}_v p(1)\ldots
p(n)}{n}|v| \geq c$$ holds. This can be shown by  the same argument
as in \cite{Len2}.
\end{Remark}
\begin{proof} (a) Let $v_k$, $k\in\NN$, be an arbitrary sequence of prefixes of
$p$ with $|v_k|\to\infty$. Then,

$$\frac{F(v_k)}{|v_k|} \leq \frac{\phi(|v_k|)}{|v_k|}$$
by definition of $\phi$. This implies
$$\limsup_k \frac{F(v_k)}{|v_k|} \leq\limsup_k
\frac{\phi(v_k)}{|v_k|} = \overline{F}.$$

Thus, it remains to show
$$\liminf_k \frac{F(v_k)}{|v_k|}\geq \overline{F}.$$
Assume the contrary. Then, going to a subsequence if necessary we
can assume without loss of generality that there exists $\delta >0$ with
$$(\clubsuit)\;\: \frac{F(v_k)}{|v_k|}\leq \overline{F} -\delta$$
for all $k\in\NN$. Let $\varepsilon >0$ be arbitrary. By definition
of $\overline{F}$ and as $|v_k|\to\infty$,  there exists then $k_0\in\NN$ with
$$(\diamond)\;\: \frac{F(w)}{|w|}\leq \overline{F} + \varepsilon$$
for all $|w|\geq |v_{k_0}|$.  Consider now an arbitrary $k\geq k_0$
and $p(1)\ldots p (N)$ for a large $N\in\NN$. Then, by $(PQ)$ (and
as $N$ is large) we can write
$$p(1)\ldots p(N) = x_0 v_k x_1  v_k ....v_k x_{m}$$
with suitable (possibly empty) words $x_k$ and at least $\frac{N c}{2 |v_k|}$ copies of $v_k$. After removing every other copy of $v_k$ we arrive at
$$p(1)\ldots p(N) = y_0 v_k y_2  v_k ....v_k y_r,$$
where now  $$|y_j|\geq |v_k| \mbox{ for all $j \in \{1,\ldots, r\}$}$$
and the number $r$ of copies of $v_k$ is still at least  $\frac{N
c}{4 |v_k|}$. Then, by subadditivity of $F$ we can calculate
\begin{eqnarray*}
\frac{F(p(1)\ldots p(N))}{N} &\leq &
\frac{r|v_k|}{N}\frac{F(v_k)}{|v_k|} + \sum_{j=1}^r
\frac{|y_j|}{N}\frac{F(y_j)}{|y_j|}\\
(\clubsuit), (\diamond) &\leq &\frac{r |v_k|}{N} (\overline{F}
-\delta) + \left(\sum_{j=1}^r \frac{|y_j|}{N}\right) (\overline{F} +
\varepsilon)\\
&\leq & \overline{F} -\frac{r |v_k|}{N}\delta + \varepsilon\\
&\leq & \overline{F} - \frac{c}{4} \delta + \varepsilon.\\
\end{eqnarray*}
Now, $\varepsilon >0$ was arbitrary. Thus, we can chose it as
$\frac{c}{8} \delta$. In this case, we infer from the preceding
computation that
\[ \frac{F(p(1)\ldots p(N))}{N} \leq \overline{F} - \frac{c}{8}
\delta \]
for all sufficiently large $N$. This contradicts the second assumption, stating that $\frac{F(p^{(n)})}{|p^{(n)}|}\to\overline{F}$ for a suitable sequence of prefixes of $p$.

\smallskip

(b) For every prefix $v$ of $p$, we define $$l_v :
\mathcal{W}(\Omega)\longrightarrow \RR, l_v (x)
:=\widetilde{\sharp_v}(x) \cdot | v |.$$ Then, $-l_v$ is subadditive and, by
assumption, the limit
\begin{equation*}
\nu(v):= \lim_{N\to \infty} \frac{l_v(p(1)\ldots p(N))}{N}.
\end{equation*}
exists. Assume now that $p$ does not satisfy (PQ). Then there exists a sequence $(v_n)$ of prefixes of $p$ with
\begin{equation*}
|v_n| \longrightarrow \infty , \text{ for } n \to \infty, \quad \text{and} \quad \sum_{n=1}^\infty \nu(v_n)<\frac{1}{2} \, .
\end{equation*}
Set $l_n:= l_{v_{n}}$ for $n\in \NN$. By the preceding
considerations  we can choose inductively for each $k\in \NN$ a
number $n(k)$, with
\begin{equation*}
\sum_{j=1}^k \frac{l_{n(j)}(w)}{|w|} <\frac{1}{2}
\end{equation*}
 for every prefix $w$ of $p$ with $|w|\geq\frac{|v_{n(k+1)}|}{2}$. Note that
 the preceding inequality  implies
\begin{equation*}
|v_{n(k)}|< \frac{|v_{n(k + 1)}|}{2}
\end{equation*}
as $\frac{l_{n(k)}(v_{n(k)})}{|v_{n(k)}|}=1$. Define the function
$l:\mathcal{W}(\Omega)\longrightarrow \RR$ by
$$l(w):= \sum_{j=1}^{\infty} l_{n(j)}(w).$$
Note that the sum is actually finite for each $w\in
\mathcal{W}(\Omega)$ (as all but finitely many of its terms vanish).
Obviously, $(-l)$ is subadditive. Thus, by assumption, the limit
$\lim_{|w|\to \infty} \frac{l(w)}{|w|}$ exists. On the other hand,
we clearly have
$$\frac{l(v_{n(k)})}{|v_{n(k)}|}\geq
\frac{l_{n(k)}(v_{n(k)})}{|v_{n(k)}|}\geq 1$$ as well as by the
induction construction
$$\frac{l(w)}{|w|}=\sum_{j=1}^k \frac{l_{n(j)}(w) }{|w|}<\frac{1}{2}$$
for any prefix $w$ of $p$  with $\frac{ |v_{n(k+1)}| }{2} \leq |w|<
|v_{n(k+1)}|$. This gives a contradiction proving (b).
\end{proof}

\begin{lemma}[Combinatorial condition for ($\gamma$)] Let $(\Omega,T)$ be a subshift and $ \omega \in \Omega $. Assume that there exists a sequences  $(w_{k})_{k\in \NN}$, of finite words of increasing length such that $\omega$ looks around the origin as
\[ \omega = ....  w_k |w_k w_k \quad \text{or as } \quad \omega = ....  w_k w_k | w_k \]
for each $ k \in \NN $. Then, for every locally constant $A : \Omega \longrightarrow \SL$ and every $v\in \RR^{2} \setminus \{ 0 \} $, at most one of the limits $ \lim_{n \to \pm \infty} \frac{1}{\lvert n \rvert}\log\|A(n,\omega^{(j)})v\| $ is negative.
\end{lemma}

\begin{proof} This follows by a variant of the so-called Gordon argument (see e.g. \cite{Dam} for background on this argument in the context of aperiodic order). We provide the details in the case that $A(\omega)$ only depends on $\omega(0)$. The general case follows after some slight modifications invoking Proposition~\ref{p:nutzlich}. Assume that $w_{k}$ exists such that $ \omega = ....  w_k |w_k w_k $ (the proof for $ \omega = ....  w_k w_k | w_k $ is similar). Define $ v_{n} := A(n,\omega) v $ for $ n \in \ZZ $. As $A$ takes values in $\SL$ we obtain from Cayley Hamilton theorem that
\[ A(|w_k|,\omega)^2 - \mbox{tr} ( A(|w_k|,\omega )) A(|w_k|,\omega) + Id = 0 \]
for all $ k \in \NN $. Here $Id$ is the $2\times 2$ identity matrix. Now, by definition of $v_n$ and our assumption we have
\[ A(|w_k|,\omega)^2 v = v_{2|w_k|} \mbox{ and } A(|w_k|,\omega) v = v_{|w_k|}.\]
So, we then obtain
\[ \|v_{2|w_k|} - \mbox{tr} ( A(|w_k|,\omega )) v_{|w_k|} + v\| = 0 \]
and after multiplication by $A(|w_k|,\omega)^{-1}$ then also
\[ \|v_{|w_k|} -\mbox{tr}( A(|w_k|,\omega )) v + v_{-|w_k|}\| = 0.\]
Given these equalities a short computation shows
\[ \max\{\|v_{-|w_k|} \|, \|v_{|w_k|}\|, \|v_{2|w_k|}\|\} \geq \frac{1}{2} \|v\|.\]
Hence, we infer that $\|v_n\|$ cannot tend to zero for $n\to\pm\infty$ if $\|v\|\neq 0$. So, in particular, we cannot have exponential decay of $\|v_n\|$ for $n\to \pm \infty$.

For the general case, where $A(\omega)$ depends on a finite word around the origin, the values of $ A( |w_k| , T^{-|w_k|} \omega) $,  $ A( |w_k| , \omega) $ and $ A( |w_k| , T^{|w_k|} \omega ) $ are not necessarily equal. More precisely, there is a finite number $N$ such that the leftmost and rightmost $N$ matrices in the products may differ. We can now consider $T^N \omega$, which satisfies
\[ A( |w_k| , T^{-|w_k|} \varrho) = A( |w_k| , \varrho) = B \cdot A( |w_k|-2N , T^{|w_k|} \varrho) \, ,\]
where $B$ is a product of $4N$ cocycle values and hence bounded. Similar to the computations above, the Cayley-Hamilton theorem yields
\[ \max \{\|v_{-|w_k|} \|, \|v_{|w_k|}\|, \|v_{2|w_k|}\|\} \geq \frac{ \|v\| }{2 \| B \| } \, .\]
This shows the claim for $T^N \omega$. By Proposition~\ref{p:nutzlich}, $\omega$ and $T^N \omega$ have the same exponential behaviour, which finishes the proof.
\end{proof}

\begin{Remark}\label{rem:absence-eigenvalues}
Note that the proof actually shows that there does not exist a $ v \in \RR^{2} \setminus \{ 0 \} $ such that $\|v_{n}\|$ tends to zero for both $ n \to \infty $ and $ n \to -\infty $. As discussed in Section~\ref{s:spectral}, the Jacobi cocycle $M^{E}$, $E\in\RR$, describes solutions $ u : \ZZ \longrightarrow \CC $ of the equation
\[ f(T^n \omega) u (n-1) + f(T^{n+1}\omega)  u (n+1) + g(T^n \omega) u (n) - E u(n) = 0 \, . \]
Thus, the Jacobi operator $H_\omega$ does not have eigenvalues if $\omega$ satisfies the condition of the Lemma.
\end{Remark}

\begin{prop}[Sufficient condition for (LSC)]\label{p:sufficient}
Let  $(\Omega,T)$ be a uniquely ergodic subshift. Then, $(\Omega,T)$
satisfies (LSC) if there exists a one-sided infinite word $p$ and
$r\in\NN$ and finite words $v^{(1)},\ldots, v^{(r)}$  and the
following conditions hold:

\begin{itemize}

\item[($\alpha'$)]  The words $\omega^{(j)}:= p^R v^{(j)} | p$ all belong to $\Omega$ and condition ($\alpha$) holds.

\item[($\beta'$)] The sequence $p$ satisfies the condition (PQ) and there is a sequence $p^{(n)}$ of prefixes of $p$
with $\lim_n F(p^{(n)}) / |p^{(n)}| \to \overline{F}$ for all
subadditive $F$.

\item[($\gamma'$)] For any $j = 1,\ldots, r$, there exists a sequence  $w_k$, $k\in \NN$, of finite words of increasing length such that $ \omega^{(j)}$ looks around the origin as
$$\omega^{(j)} = ....  w_k |w_k w_k  \mbox{ or as  } \omega^{(j)} = ....  w_k w_k | w_k ...  .$$
\end{itemize}
\end{prop}
\begin{proof} Clearly ($(\alpha')$ implies validity of ($\alpha$).
Thus, it remains to  show that the subshift satisfies ($\beta$) and
($\gamma$):

\smallskip

\textit{The subshift satisfies ($\beta$):} We first  gather a few
simple consequences of ($\beta'$) and ($\gamma'$). Clearly, the
subshift is palindromic (i.e. $w\in \mathcal{W}(\Omega)$ implies
$w^R\in\mathcal{W}(\Omega)$). Indeed, by minimality any $w\in
\mathcal{W}(\Omega)$ must appear in $p$ infinitely often and then
$w^R$ appears in $p^R$ infinitely often. Moreover, with $F :
\mathcal{W}(\Omega)\longrightarrow \RR$ also the function
$$F^R :\mathcal{W}(\Omega) \longrightarrow \RR, F^R (w) := F(w^R),$$
is subadditive. Hence, by (PQ) and Lemma~\ref{lem:CombCondB} (a)  for any
subadditive $F$ also the limit
$$\lim_{n\to\infty} \frac{F(p(n)\ldots p(1))}{n} $$
exists. If $F$ is induced by a locally constant cocycle $A$ via
$$F(w) = \max\{ \log\|A(|w|,\omega)\| : \omega \in \Omega \mbox{
with } \omega(1)\ldots \omega (|w|) = w\}$$ it is easy to see from
Proposition \ref{p:nutzlich} that
$$\Lambda_+:=\lim_{n\to \infty} \frac{ \log\|A(n,\omega^{(j)})\|}{n} =
\lim_{n\to \infty} \frac{F(p(1)\ldots p(n))}{n}$$ and
$$\Lambda_-:=\lim_{n\to \infty} \frac{ \log\|A(-n,\omega^{(j)})\|}{n} =
\lim_{n\to \infty} \frac{F(p(n)\ldots p(1))}{n}$$ both exist. Moreover, since $\omega^{(j)}= p^R v^{(j)} | p$ and ($\gamma'$) yield
\[ p( |w_k| - |v^{(j)}|) \ldots p( 1 )  =  p( 1 ) \ldots p( |w_k| - |v^{(j)}|) \]
and the words $w_k$ get arbitrary large, we must have $\Lambda_+ = \Lambda_-$. This gives ($\beta$).

\smallskip

\textit{The subshift satisfies ($\gamma$):} This follows from the previous lemma.
\end{proof}


\section{Simple Toeplitz subshifts satisfy (LSC)}\label{s:simple}
In this section we consider simple Toeplitz subshifts. It is well-known that these are aperiodic. We show that they satisfy (LSC). Then, by Theorem \ref{l:main-technical} every locally constant
cocycle over such a subshift is uniform and by Theorem
\ref{t:main-abstract} the associated Jacobi operators have
Cantor spectrum. As discussed in the introduction, this generalizes
the results of \cite{LiuQu} and can be seen as the main results of
this article. We only have to show that simple Toeplitz subshifts
satisfy the conditions of Proposition \ref{p:sufficient}.

\bigskip

First, we recall the definition of a (simple) Toeplitz subshift:
Let
$ \omega \in \mathcal{A}^{\ZZ} $ be a two-sided infinite word such
that for all $ j \in \ZZ $ there exists $ p \in \NN $ with $ \omega(
j+kp ) = \omega( j ) $ for all $ k \in \ZZ $. The orbit closure of
such an $\omega$ under the shift action is called a Toeplitz
subshift. For more details, we refer the reader to the survey
\cite{Downarowicz} and the references therein. Here, we will only
consider the subclass of so called simple Toeplitz subshifts, which
exhibit additional structure and are defined as follows: Let
$(a_{k})_{k} \in \mathcal{A}^{\NN \cup \{ 0 \}} $ be a sequence of
letters and $ (n_{k})_{k} \in ( \NN \setminus \{1\} )^{\NN \cup \{ 0
\}}$ a sequence of period lengths that are greater or equal two.
Those sequences are called \emph{coding sequences} of a simple
Toeplitz subshift. Let $\widetilde{\mathcal{A}}$ denote the
\emph{eventual alphabet}, that is, the set of letters that appear
infinitely often in $(a_{k})$. In the following, we will always
assume $\# \widetilde{\mathcal{A}} \geq 2$ in order to exclude
periodic words. Moreover, we assume $ a_{k+1} \neq a_{k} $, since
consecutive occurrences of the same letter can be expressed as a
single occurrence if $n_{k}$ is increased accordingly. We use
$\widetilde{K}$ to denote a number such that $ a_{k} \in
\widetilde{\mathcal{A}} $ holds for all $ k \geq \widetilde{K}$. We
define the subshift from palindromic blocks: Let
\[ \block{-1} := \epsilon \qquad \text{and} \qquad \block{k+1} := \block{k} a_{k+1} \block{k} \hdots \block{k} a_{k+1} \block{k} \]
with $n_{k+1}$-many $\block{k}$-blocks and $(n_{k+1}-1)$-many
$a_{k+1}$'s. Clearly, $\block{k}$ is a palindrome for every $k \in
\NN \cup \{ 0 \}$. Note that $ \length{\block{k+1}} + 1= n_{k+1}
(\length{\block{k}} + 1)$ holds for all $k \geq -1 $, which implies
$ \length{\block{k+1}} + 1 = \prod_{j=0}^{k+1} n_{j}$. Moreover,
$\block{k}$ converges to a one-sided infinite word $\block{\infty}
:= \lim_{k \to \infty} \block{k}$. Define $ \Omega := \{ \omega \in
\mathcal{A}^{\ZZ} : \mathcal{W}( \omega ) \subseteq
\mathcal{W}(\block{\infty}) \} $.

Alternatively, simple Toeplitz subshifts can also be defined via a ``hole filling procedure''. As this is instructive, we include some details next (see Section \ref{s:family} and \cite{LiuQu} as well): Let $ ? \notin \mathcal{A} $ be an additional letter which represents the ``hole''. In addition to the sequences $ (a_{k}) $ and $ (n_{k}) $ from above, let $ (r_{k})_{k \in \NN \cup \{ 0 \}} $ be a sequence of integers with $ 0 \leq r_{k} < n_{k} $. We define the two-sided infinite, periodic words
\[ (a_{k}^{n_{k}-1}?)^{\infty} := \hdots a_{k} \hdots a_{k} ? \underbrace{a_{k} \hdots a_{k}}_{n_{k}-1 \text{-times}} ? a_{k} \hdots a_{k} ? \hdots \]
with period length $n_{k}$ and holes at $ n_{k} \ZZ + r_{k} $. We now insert $(a_{1}^{n_{1}-1}?)^{\infty}$ into the holes of $(a_{0}^{n_{0}-1}?)^{\infty}$, that is, we define a new word $ (a_{0}^{n_{0}-1}?)^{\infty} \triangleleft (a_{1}^{n_{1}-1}?)^{\infty} $ by
\[ ((a_{0}^{n_{0}-1}?)^{\infty} \triangleleft (a_{1}^{n_{1}-1}?)^{\infty})(j) := \begin{cases}
(a_{0}^{n_{0}-1}?)^{\infty}(j) & \text{for } j \notin n_{0} \ZZ + r_{0}\\
(a_{1}^{n_{1}-1}?)^{\infty}( \frac{j-r_{0}}{n_{0}} ) & \text{for } j \in n_{0} \ZZ + r_{0}
\end{cases} \, . \]
By inserting $(a_{2}^{n_{2}-1}?)^{\infty}$ into the holes of the obtained word, then inserting $(a_{3}^{n_{3}-1}?)^{\infty}$, and so on, we obtain a sequence of two-sided infinite words
\[ \omega_{k} := (a_{0}^{n_{0}-1}?)^{\infty} \triangleleft (a_{1}^{n_{1}-1}?)^{\infty} \triangleleft (a_{2}^{n_{2}-1}?)^{\infty} \triangleleft \hdots  \triangleleft (a_{k}^{n_{k}-1}?)^{\infty} \, . \]
When we take the limit $ \omega_{\infty} := \lim_{k \to \infty} \omega_{k} $ in $ ( \mathcal{A} \cup \{ ? \} )^{\ZZ} $, there is at most one position where $ \omega_{\infty} $ has a hole. If such a position exists, then we fill the hole by an arbitrary letter from $ \widetilde{\mathcal{A}} $. If now $\omega \in \mathcal{A}^{\ZZ} $ denotes the word that was obtained this way, we define the simple Toeplitz subshift as $ \Omega := \overline{ \{ T^{k} \omega : k \in \NN \} } $. It is equal to the subshift that was defined above in terms of $\block{\infty}$ (see for example \cite{Sell_SimpToepCombinat}, Proposition~2.6).

It was shown in \cite{LiuQu}, Corollary~2.1 that every simple
Toeplitz subshift $(\Omega_{\omega}, T)$ is minimal and uniquely
ergodic. In addition, $ \# \widetilde{\mathcal{A}} \geq 2 $ implies
that every simple Toeplitz word defined by $(a_{k})_{k}$ is
non-periodic (see for example \cite{Sell_SimpToepCombinat},
Proposition~2.2). Conversely, $\# \widetilde{\mathcal{A}} =1 $
clearly gives periodicity of the subshift. In Proposition~2.4 in
\cite{LiuQu} it  was shown that every element in the subshift can be
obtained by the hole filling procedure with the same coding
sequences $(a_{k})$ and $(n_{k}) $. From this, it easily follows
that for every $k \in \NN \cup \{ 0 \}$ and every $ \omega \in
\Omega $, there is a unique decomposition of $\omega$ in the form
\[ \omega = \hdots \block{k} \star \block{k} \star \block{k} \star \block{k} \hdots \quad , \]
where $ \star $ denotes elements from $ \{ a_{j} : j \geq k+1 \} $.

To prove that simple Toeplitz subshifts satisfy (LSC), we will show
that they satisfy the sufficient conditions of
Proposition~\ref{p:sufficient}. This will be done in two  steps. In
the first one, we discuss the words $\omega^{(j)}$ and study their
combinatorial properties. In the second one, we treat asymptotic averages of subadditive functions.

We start with the discussion of the words $\omega^{(j)}$.
Consider the one-sided infinite word $ p = \block{\infty} $.
With $ r = \# \widetilde{\mathcal{A}} $,
we write $ \widetilde{\mathcal{A}} = \{ a^{(1)}, \hdots , a^{(r)} \} $
and define the words $v^{(j)}$ of length one by $v^{(j)} = a^{(j)}$. Then for every $j$, there are arbitrary large $k$ with $a_{k+1} = a^{(j)}$. Thus the words $ \block{k} a^{(j)} \block{k} $ occur in the subshift and hence all
$$\omega^{(j)}:= p^R v^{(j)} | p
= (\block{\infty})^R a^{(j)} | \block{\infty} $$
belong to $\Omega$.
Next we check that all sufficiently long finite words occur in some
$\omega^{(j)}$ close to the origin. Recall that $\widetilde{K}$
denotes a number such that $ a_{k} \in \widetilde{\mathcal{A}} $
holds for all $ k \geq \widetilde{K}$.

\begin{prop}[Occurrence of words around the origin]
For all $ L \geq \length{\block{\widetilde{K}}} $ and all $ u \in
\mathcal{W}(\Omega)_{L} $, there are $ j \in \{ 1 , \hdots , r \} $
and $ k \in \{ 1, \ldots, L\}$ such that $ u = \omega^{(j)}(-k)
\ldots \omega^{(j)}(-k + L-1) $ holds.
\end{prop}

\begin{proof}
For every $ u \in \mathcal{W}(\Omega)_{L} $ there exists an $\omega
\in \Omega$ such that $u$ occurs in $\omega$. Let $k$ denote the
unique number such that $ \length{\block{k}} < L \leq
\length{\block{k+1}} $ holds and decompose $\omega$ as $ \omega =
\hdots \block{k+1} \star \block{k+1} \star \block{k+1} \hdots $ with
single letters $ \star \in \{ a_{j} : j \geq k+2 \} $. We
distinguish two cases: Firstly, if $u$ is completely contained in
$\block{k+1} = \block{k} a_{k+1} \block{k} \hdots \block{k} $, then
$u$ contains at least once the single letter $a_{k+1}$. Choose $j$
such that $ a^{(j)} = a_{k+1} $ holds. Around the origin,
$\omega^{(j)}$ has the form
\[ \hdots \block{k+1} a^{(j)} | \block{k+1} \hdots = \hdots \block{k} a_{k+1} \block{k} \hdots \block{k} a_{k+1} | \block{k} a_{k+1} \block{k} \hdots \block{k} \hdots \quad . \]
Now the claim follows by aligning an occurrence of $a_{k+1}$ in $u$
with $\omega^{(j)}(-1)$. Secondly, if $u$ is not contained in a
single $\block{k+1}$-block, then there is a letter $ a \in \{ a_{j}
: j \geq k+2 \} \subseteq \widetilde{\mathcal{A}} $ such that $u$ is
contained in $ \block{k+1} a \block{k+1} $ and $a$ is contained in
$u$. Choose $j$ such that $ a^{(j)} = a $. Around the origin,
$\omega^{(j)}$ is of the form $ \hdots \block{k+1} a^{(j)} |
\block{k+1} \hdots $ and thus $ u $ occurs in $\omega^{(j)}$ as
claimed.
\end{proof}

Finally, we check property ($\gamma'$) of the sufficient conditions
in Proposition~\ref{p:sufficient}:

\begin{prop}[Occurrences of three blocks] For any $j=1,\ldots r$, there exists a sequence $w_i$ of words of increasing length such that $\omega^{(j)} = ...w_i w_i |w_i...$ holds for all $i \in \NN$.
\end{prop}
\begin{proof} For every $ a^{(j)} \in \widetilde{\mathcal{A}}$, there is an
increasing sequence $k_{i}$ with $ a_{k_{i}} = a^{(j)} $. Consider
the finite words $ w_{i} = \block{k_{i}-1} a_{k_{i}} $. Since every $\block{k}$ is a prefix as well as a suffix of $\block{k+1}$, the element $\omega^{(j)}$ looks around the origin like
\[ \hdots \block{k_{i}} a^{(j)} | \block{k_{i}} \hdots = \hdots \block{k_{i}-1} a_{k_{i}} \block{k_{i}-1} a_{k_{i}} | \block{k_{i}-1} a_{k_{i}} \hdots = \hdots w_{i} w_{i} | w_{i} \; . \qedhere \]
\end{proof}

After having discussed the words $\omega^{(j)}$, we now discuss averages of subadditive functions.  First we show that the blocks
$\block{k}$ are prefixes with the limit property that is required in
condition $(\beta ')$ in Proposition~\ref{p:sufficient}.

\begin{prop}
Let $F : \mathcal{W}(\Omega)\longrightarrow \RR$ be a subadditive
function and let
 $$ \overline{F} := \lim_{L \to \infty} \max_{x :
\length{x} = L} \frac{F( x )}{L}.$$
 Then $\lim_{k \to \infty}
\frac{F( \block{k} )}{\length{ \block{k} }} = \overline{F}$ holds.
\end{prop}

\begin{proof}
By definition of $\overline{F}$, we have $ \limsup_{k \to \infty}
\frac{F( \block{k} )}{\length{ \block{k} }} \leq \overline{F} $.
Therefore, it only remains to show that $ \liminf_{k \to \infty}
\frac{F( \block{k} )}{\length{ \block{k} }} \geq \overline{F} $
holds: Let $ D := \max \{ F( a ) : a \in \mathcal{A} \} $ and fix an
arbitrary $ k \in \NN $. Let $ L \geq \length{\block{k}} $ and let
$x$ be an arbitrary word of length $L$. Then $x$ is contained in
some element $\omega \in \Omega$ and every $\omega$ can be
decomposed as $ \omega = \hdots \block{k} \star \block{k} \star
\block{k} \hdots $ with single letters $ \star \in \mathcal{A} $. We
obtain $ x = u \star \block{k} \star \block{k} \hdots \block{k}
\star v $, where $u$ is a suffix and $v$ is a prefix of $\block{k}$.
Note that we have at most $\frac{L}{\length{\block{k}}}$ blocks
$\block{k}$ and $\frac{L}{\length{\block{k}}} + 1 <
\frac{2L}{\length{\block{k}}}$ single letters $\star$ in $x$. Hence
we obtain
\begin{align*}
\frac{F( x )}{L}
& \leq \frac{F( u )}{L} + \frac{\frac{2L}{\length{\block{k}}} \cdot D}{L} + \frac{\frac{L}{\length{\block{k}}} \cdot F( \block{k} )}{L} + \frac{F( v )}{L}\\
&\leq \frac{\length{\block{k}} D}{L} + \frac{2
D}{\length{\block{k}}} + \frac{F( \block{k} )}{\length{\block{k}}} +
\frac{\length{\block{k}} D}{L} \, .
\end{align*}
Since $x$ was arbitrary, the above yields for every $k$ and every $L
\geq \length{\block{k}}$ the inequality
\[ \max_{x : \length{x} = L} \frac{F( x )}{L} \leq \frac{\length{\block{k}} D}{L} + \frac{2 D}{\length{\block{k}}} + \frac{F( \block{k} )}{\length{\block{k}}} + \frac{\length{\block{k}} D}{L} \, . \]
In particular we can take the limit $ L \to \infty$ and obtain for
every $k$
\[\overline{F} \leq \frac{2 D}{\length{\block{k}}} + \frac{F( \block{k} )}{\length{\block{k}}} \quad \text{and thus } \quad \overline{F} \leq \liminf_{k \to \infty} \frac{F( \block{k} )}{\length{\block{k}}} \, . \qedhere\]
\end{proof}

To prove (LSC) for simple Toeplitz subshifts, it only remains to
show that condition (PQ) is satisfied.

\begin{prop}[Validity of (PQ)]
For every prefix $v$ of $\block{\infty}$, the inequality
\[ \liminf_{L \to \infty} \frac{ \widetilde{\#}_{v} \block{\infty}(1) \hdots \block{\infty}(L) }{L} \length{v} \geq \frac{1}{8} \]
holds, that is, the sequence $\block{\infty}$ satisfies (PQ).
\end{prop}

\begin{proof}
Let $v$ be a prefix of $\block{\infty}$. Let $K$ be such that $
\length{ \block{K} } < \length{v} \leq \length{ \block{K+1} } $
holds and let $m$ be such that $ m( \length{ \block{K} } + 1 ) \leq
\length{v} < (m+1)( \length{ \block{K} } + 1 ) $ holds. Note that
this implies $ 1 \leq m < n_{K+1} $. First we compute an auxiliary
result. Clearly
\[ \widetilde{\sharp}_{v} \block{K+1} \frac{ \length{v} }{\length{ \block{K+1} } + 1} \geq \left\lfloor \frac{n_{K+1}}{m+1} \right\rfloor \frac{ m( \length{ \block{K} } +1 ) }{n_{K+1}( \length{ \block{K} } +1 )} = \left\lfloor \frac{n_{K+1}}{m+1} \right\rfloor \frac{ m }{n_{K+1} } \]
holds. To see that this term is bounded away from zero, we distinguish
two cases:

\begin{itemize}
\item{If $ \frac{n_{K+1} }{m+1} < 2 $ holds, then we obtain $ \left\lfloor \frac{n_{K+1}}{m+1} \right\rfloor \frac{ m }{n_{K+1} } = 1 \cdot \frac{m}{m+1} \frac{m+1}{n_{K+1}} > \frac{1}{4} $.}
\item{If $ \frac{n_{K+1} }{m+1} \geq 2 $ holds, then we obtain $ \left\lfloor \frac{n_{K+1}}{m+1} \right\rfloor \frac{ m }{n_{K+1} } > \left( \frac{n_{K+1}}{m+1} - 1 \right) \frac{ m }{n_{K+1} } = \big( 1 - \frac{m+1}{n_{K+1}} \big) \frac{m}{m+1} \geq \frac{1}{4} $.}
\end{itemize}
It is now easy to provide the necessary bound on
$\widetilde{\sharp}_v \block{\infty}(1) \ldots \block{\infty}(L)$ from below:
\begin{align*}
\widetilde{\sharp}_{v} \block{\infty}(1) \ldots \block{\infty}(L) \cdot \frac{
\length{v} }{L}
&\geq \widetilde{\sharp}_{v}\block{K+1} \cdot \frac{ \length{v} }{ \length{\block{K+1}} + 1 } \cdot \widetilde{\sharp}_{\block{K+1}} \block{\infty}(1) \ldots \block{\infty}(L) \cdot \frac{ \length{\block{K+1}} + 1 }{L} \\
&> \frac{1}{4} \cdot \left( \frac{ L }{ \length{\block{K+1}} + 1 } - 1 \right) \cdot \frac{ \length{\block{K+1}}+1  }{L} \\
&= \frac{1}{4} \cdot \left(1 - \frac{ \length{\block{K+1}}+1  }{L}
\right)
\end{align*}
For all sufficiently large $L$ we have $ \frac{
\length{\block{K+1}}+1  }{L} \leq \frac{1}{2} $, which yields the
claim.
\end{proof}

We summarize the content of the preceding propositions in the next
theorem.

\begin{theorem}\label{thm-simple-toeplitz-lsc}
Any simple Toeplitz subshift satisfies (LSC). In particular, all
locally constant cocycles over simple Toeplitz subshifts are
uniform.
\end{theorem}
\begin{proof} The preceding propositions show that the assumptions
of Proposition \ref{p:sufficient} are satisfied. This proves the
first statement. The last statement is an immediate consequence of
the first statement and Theorem \ref{t:main-abstract}.
\end{proof}

\begin{Remark}[Purely singular continuous spectrum]
When combined with Theorem \ref{t:main-abstract}, the previous
theorem  implies that  the spectrum of an aperiodic  Jacobi operator
associated to a simple Toeplitz subshift is a Cantor set of Lebesgue
measure zero. Thus, the spectrum is singular. Moreover, it can be
shown that, for almost all $ \omega \in \Omega $ with respect to the
unique ergodic probability measure, $H_{\omega}$ does not have
eigenvalues \cite{Sell-PhD}. If $n_{k} \geq 4$ holds for all $k \geq
0$, then the spectrum of $H_{\omega}$ is actually purely singular
continuous for all $ \omega \in \Omega$, that is, no $H_{\omega}$
has eigenvalues (see also Theorem~1.3 in \cite{LiuQu}).
\end{Remark}

As mentioned above, simple Toeplitz subshifts with eventual alphabet
containing at least two letters  are aperiodic. This aperiodicity is
stable under taking (suitable) morphisms. This will be relevant in
the application to Jacobi operators. Specifically, we will need the
following proposition.

\begin{prop}\label{prop:aperiodicity} Let $(\Omega,T)$ be a simple
Toeplitz subshift with eventual alphabet $ \widetilde{\mathcal{A}}$
containing at least two letters. Let $\mathcal{B}$ be an arbitrary
finite set and $\Phi : {\mathcal{A}} \longrightarrow \mathcal{B}$
not constant on $\widetilde{\mathcal{A}}$. Define for $\omega
\in\Omega$ the word $\Phi (\omega)$ in $\mathcal{B}^\ZZ$ via
$$\Phi(\omega)(n) :=\Phi(\omega(n)).$$
Then, $(\Phi(\Omega),T)$ is an aperiodic simple Toeplitz subshift.
\end{prop}
\begin{proof} Clearly, $\Phi (\Omega)$ is simple Toeplitz with
eventual alphabet $\Phi(\widetilde{\mathcal{A}})$. By assumption on
$\Phi$ this alphabet has at least two elements and the statement
follows.
\end{proof}

\begin{Remark}[Boshernitzan condition and simple Toeplitz subshifts]\label{rem-char-Bosh}
As  mentioned in the introduction of this article not all simple
Toeplitz satisfy Boshernitzan condition. Indeed, an explicit
characterization of the Toeplitz subshifts satisfying this condition
is given in \cite{LiuQu}. In Section \ref{s:family}, it will be seen
that the class of simple Toeplitz subshifts where $n_k$ is a power of two for all $k\in
\NN\cup\{0\}$, will be of particular interest to us. When restricted
to this case one obtains that validity of Boshernitzan condition is
equivalent to existence of a natural number  $C$ and a sequence
$(t_r)$ of natural numbers with $t_r \to\infty$ and
$\{a_{t_r},\ldots, a_{t_r + C}\} = \{a_{s} : s\geq t_r\}$ for all
$r$ (Corollary 6.5 in \cite{Sell_SimpToepCombinat}).
\end{Remark}

\section{Sturmian Subshifts satisfy (LSC)}\label{s:sturmian}
In this section we show that Sturmian subshifts  satisfy (LSC). By
our main results the associated Jacobi operators then have Cantor
spectrum of Lebesgue measure zero. Of course, this is well known, \cite{BIST}, but we include the discussion for completeness.

\bigskip

We will show that Sturmian subshifts satisfy the conditions of
Proposition \ref{p:sufficient}.  Let us first recall how Sturmian
subshifts are defined. Here we freely follow \cite{Ber} (see
\cite{Len4} as well).  Let $\alpha$ be an irrational number with
continued fraction expansion
\begin{equation*} \alpha=[a_1,a_2,\ldots]=\cfrac{1}{a_1+
\cfrac{1}{a_2+ \cfrac{1}{a_3 + \cdots}}}.
\end{equation*}
Define recursively the words $s_n$, $n\in \NN \cup \{ 0 \} $ over the alphabet
$\{0,1\}$ via
\begin{equation*}
s_{-1} := 1, \;\:\;\;\:\; s_0 := 0, \;\:\;\;\:\; s_1 := s_0^{a_1 -
1} s_{-1}, \;\:\;\;\:\; s_n := s_{n-1}^{a_n} s_{n-2}, \;\:\; n \ge
2.
\end{equation*}

Define
$$\mathcal{W}(\alpha):=\cup_{n} \mathcal{W}(s_n)$$ and set
$$\Omega(\alpha):=\{ \omega\in\{0,1\}^\ZZ :
\mathcal{W}(\omega)\subset \mathcal{W}(\alpha)\}.$$ Then,
$\Omega(\alpha)$ is invariant under $T$ and closed, and
$(\Omega(\alpha),T)$ is called the \emph{Sturmian subshift} with
\emph{rotation number} $\alpha$.

There exist palindromes $\pi_n$, $n \ge 2$ with
\begin{equation*}
  s_{2n} = \pi_{2n}10,
\end{equation*}
\begin{equation*}
\label{odd}
  s_{2n+1} = \pi_{2n+1}01.
\end{equation*}
Moreover, clearly  $s_{n-1}$ is a prefix of $s_n$ for $n\geq 2$ and
we can therefore define the 'right limit'
\begin{equation*}
  \label{Rechtsgrenzwert}
c_\alpha := \lim_{n \rightarrow \infty} s_n.
\end{equation*}
Analogously, for  $n\geq 2$ also $s_{n-1}$ is a suffix of $s_{n+1}$.
This gives existence of the 'left-limits'
\begin{equation*}  \label{Linksgrenzwert}
\dalpha := \lim_{n \rightarrow \infty} s_{2n},\hspace{5ex} \ealpha:=
\lim_{n \rightarrow \infty} s_{2n+1}.
\end{equation*}
Define $u_\alpha$ to be the two-sided infinite word which agrees
with $\calpha$ on $\NN$ and with $\dalpha$ on $-\NN\cup\{0\}$ and
define $v_\alpha$ to be the  two-sided infinite word which agrees
with $\calpha$ on $\NN$ and with $\ealpha$ on $-\NN\cup\{0\}$. Then,
both $u_\alpha$ and $\valpha$ belong to $\Omega(\alpha)$.

\begin{lemma} Every Sturmian subshift satisfies (LSC) with $r =2$ and
$\omega^{(1)} = \ualpha$ und $\omega^{(2)} = \valpha$.
\end{lemma}
\begin{proof}
We show that the assumptions of Proposition \ref{p:sufficient} are
satisfied with $p =c_\alpha$, $v^{(1)} = 01$ and $v^{(2)} = 10$:

\smallskip

\textit{Condition ($\alpha'$) is satisfied:}  As discussed above
$s_n=\pi_n a b$ with $a,b\in\{0,1\}$ for $n\geq 2$. From the
definition of $\ualpha$ and $\valpha$ we then directly  infer that
$\omega^{(1)} = p^R v^{(1)} p$ and $\omega^{(2)} = p^R v^{(2)} p$. So, it remains to show the statement on occurrences of
words around the origin in $\omega^{(1)}$ and $\omega^{(2)}$. This
is not hard to see and is given e.g. in  Lemma 6.6 of \cite{Len4}.
This finishes the proof of  ($\alpha'$).

\smallskip

\textit{Condition ($\beta'$) is satisfied:} This follows from
Theorem 11 of \cite{Len4} (but it is not hard to give a direct
proof).

\smallskip

\textit{Condition $(\gamma')$ is satisfied:} A short computation
shows that $(s_n s_{n+1})^{*} = (s_{n+1} s_n)^*$ for $n\geq 3$,
where the $*$ indicates that the last two letters are removed. From
this we easily see that
$$ \omega^{(1)} = .... s_{2n+1} | s_{2n+1} s_{2n+1} ... \qquad \text{and} \qquad \omega^{(2)} = .... s_{2n} | s_{2n} s_{2n} ... $$
for each $n\geq 1$.
\end{proof}

\section{Spectra of Schreier graphs of spinal groups}\label{s:family}
In this section we will apply our results to an interesting class of
examples that served as an initial motivation to this work. Our
examples are from a seemingly different context, that of finitely
generated groups of automorphisms of rooted trees. These groups came
into light chiefly after the discovery in this class of first
examples of groups on intermediate growth \cite{Gr84}. The
construction provides an uncountable family $\{G_\xi\}$,
$\xi\in\{0,1,2\}^\mathbb N$ of groups of automorphisms of the
infinite binary tree with growth strictly between polynomial and
exponential, if the sequence $\xi$ is not eventually constant. Each
group $G_\xi$ is generated by $4$ involutions.
One example that has been particularly well studied and enjoys
additional nice properties, like being generated by a finite
automaton, is the so-called ``first Grigorchuk's group'' $\mathcal G =
\langle a,b,c,d\rangle $ (see e.g. Chap. VIII in \cite{dlH}). In the
family $\{G_\xi\}$ it corresponds to  $\xi =(012)^\infty$.

An action of a group $G$ by automorphisms on a rooted spherically
homogeneous tree $T$  (that we will always assume transitive on the
levels of the tree) extends to an action by homeomorphisms on the
boundary of the tree $\partial T$. If a finite generating set $S$ is
chosen in $G$, then the orbital partition corresponding to the
action of $G$ on $\partial T$ gives rise to a map $\Phi$ from
$\partial T$ to the space of (isomorphism classes of) rooted regular
graphs with edges labelled by elements of $S$. The closure of its
image with isolated points removed is called the \emph{space of (orbital)
Schreier graphs} of $G$ (with respect to $S$) and is denoted
$Sch(G)$. The group acts on it by changing the root. Then, the
dynamical  system $(Sch(G), G)$ is minimal and uniquely ergodic
(with unique invariant probability measure given by  the pushfoward
of the uniform measure on the boundary of the tree under $\Phi$).
Schreier graphs are interesting objects in their own right and serve
as a useful tool in the study of the group. More generally, given a
finitely generated group $G$ and a subset $S\subset G$, a \emph{Schreier
graph} can be defined for any transitive action of $G$ on a set $X$:
the vertex set of the graph is the set $X$ and the set of (labelled
oriented) edges is $\{(x,s\cdot x)|x\in X,s\in S\}$. The graph is
connected if and only if $S$ generates $G$. It is regular of degree
$|S|$. As a particular case, if $S=S^{-1}$ and the action of $G$ on
$X$ is free, we get the Cayley graph of $G$ with respect to $S$.

For a finitely generated group $G$ and a chosen finite symmetric
generating set $S$, the corresponding Cayley graph and Schreier
graphs for natural group actions present in particular an
interesting class of examples in spectral graph theory that
investigates the spectra of Laplacians acting on the $l^2$-space on
the vertices of the graph.  Here, we prefer to consider Markov
operators
$$M=\sum_{s\in S} p_s s$$
with $p_s>0$, $p_s=p_{s^{-1}}$, for all $s\in S$,  and $\sum_{s \in
S} p_s =1$ (or, more generally, $\sum_{s \in S} p_s \leq 1$). The
corresponding Laplacian is then just $1-M$. Clearly, the spectral
type does not change if we add constants to the operator. Thus, we
can deal with Laplacians as well. The question about spectral type
of Schreier graphs and Cayley graphs of finitely generated groups is
in general widely open.

The paper \cite{BG} was one of the first to address the spectral
theory of Schreier graphs of groups acting on rooted trees. It
presents an example of a group whose orbital Scherier graphs for the
action on the boundary of the tree have Cantor spectrum of Lebesgue
measure $0$ and a countable set of points, and another example, the
group $\mathcal G$ mentioned above, where this spectrum is a union
of two disjoint intervals. They only considered the isotropic Markov
operator, i.e., with $p_s = 1/|S|$ for all $s\in S$. Note that, as
the groups in question are amenable, the spectrum as a set is an
invariant of the space of Schreier graphs and does not depend on a
particular orbit.

The construction from \cite{Gr84} has been generalized in a number
of ways. One way, initially suggested in \cite{BS}, leads to the
so-called spinal groups, of which we describe here one particular
construction. For each $d\geq 2, m\geq 1$, we consider an
uncountable family of groups $\{G_\xi\}$ acting by automorphisms on
the infinite $d$-regualr rooted tree $T_d$, with $\xi\in \Xi_{d,m}
\subset Epi(B,A)^\mathbb N$, where $A=\mathbb Z/d\mathbb Z$,
$B=(\mathbb Z/d\mathbb Z)^m$ and $\Xi_{d,m}$ consists of all
infinite sequences of epimorphisms that have trivial intersection of
kernels over any tail.  For $\xi\in \Xi_{d,m}$, the group $G_\xi$ is
generated by the automorphism $a$ that cyclically permutes the
branches at the root of the tree and a copy $B_\xi$ of $B$ in
$\Aut(T_d)$. The action of the elements from  $ B_\xi$ on the tree
can be described as follows: Any $b_\xi \in B_\xi$ acts trivially
everywhere but on the subtrees rooted at the vertices of the rightmost infinite ray in
the tree. In the subtree rooted in the vertex at the $r$-th level of
$T_d$, it acts by permuting the branches at the root of the subtree
as $\xi_r(b)$ (see \cite{NP,GNP} for a more detailed description of
this specific class of spinal groups). For $d=2,m=2$ we recover the
family from \cite{Gr84}. More generally, all of these examples with
$d=2$ are of intermediate growth.

 For $d=2$, the Schreier graphs of a spinal group $G_\xi$ with respect to the generating set $S_\xi=\{a\}\cup B_\xi \setminus \{ \text{id} \}$ have the similar structure: they are lines with loops and multiple edges.
The linear structure of Schreier graphs allows to associate a
subshift to the dynamical system $(Sch(G_\xi), G_\xi)$. It was shown
in \cite{GLN} that the Markov operators on the Schreier graphs
become then unitary equivalent to the Schr\"odinger operators on the
associated subshift. It is also shown there that the subshift
associated with the first example $\mathcal G$ is linearly
repetitive, and hence the Cantor spectrum of Lebesgue measure $0$
theorem for Schr\"odinger operators on linearly repetitive subshifts
applies and yields new information about the Laplacian spectrum on
the orbital Schreier graphs for the action of $\mathcal G$ on the
boundary of $T_2$. Namely, while the periodic potential
corresponds to the isotropic Markov operator whose spectrum was
already known from \cite{BG}, the case of aperiodic potential shows
that the spectrum of the anisotropic Markov operator ($p_b, p_c,
p_d$ not all equal) is a Cantor set of Lebesgue measure $0$. This
has interesting consequences, as it implies  in particular Cantor
spectrum of Lebesgue measure $0$ for the isotropic Markov operator
on the Schreier graph of $\mathcal G$ with the minimal generating
set $\{a,b,c\}$, see \cite{GNP}.

The question arises as to which extent our results hold for other
spinal groups acting on the binary tree. It turns out that very few
groups in the families $\{G_\xi\}$, $\xi\in \Xi_{2,m}$, $m\geq 2$
give rise to  linearly repetitive subshifts. It follows from
\cite{NP} that the set of $\xi\in \Xi_{d,m}$ with linearly
repetitive Schreier graphs is of measure $0$ with respect to the
Bernoulli measure on the set of parameters. The more general
Boshernitzan condition which also would be enough to imply the
Cantor spectrum in the aperiodic case is verified on Schreier graphs
of groups forming a set of measure $1$ in the space of parameters
\cite{NP}, but not all $G_\xi$ satisfy it either (compare Remark
\ref{rem-char-Bosh} as well). Here we prove that for all $m\geq 2$
and all $\xi \in \Xi_{2,m}$, the subshift defined by  Schreier
graphs of $G_\xi$ is simple Toeplitz. The results from the previous
sections then allow us to extend the result from \cite{GLN} and to
deduce the Cantor spectrum of Lebesgue measure $0$ for anisotropic
Markov operators on orbital Schreier graphs of an arbitrary spinal
group $G_\xi, \xi \in \Xi_{2,m}$, $m\geq 2$.

Indeed, consider a new alphabet $\mathcal A = \{a\}\cup\{
\alpha_\phi | \phi\in Epi(B,A)\}$, so that a letter in the alphabet
$\mathcal A$ is either $a$ or represents $B\setminus Ker(\phi)$,
$\phi\in  Epi(B,A)$, a possible set of labels on a multi-edge
between two vertices in the Schreier graph. Consider the one-sided
infinite sequence $\eta$ in this alphabet that we read on the
Schreier graph rooted at the boundary point $1^\infty$, and
associate to $G_\xi$ the corresponding two-sided subshift. We now
observe that, as the infinite, rooted at $1^\infty$ Schreier graph
is the limit, as $n\rightarrow\infty$, of the finite Schreier graphs
on the vertex set of the $n$-th level of  $T_d$ rooted at $1^n$, the
sequence $\eta$ is the limit of the words in the alphabet $\mathcal
A$ read on these finite Schreier graphs starting from the root
$1^n$. The structure of the finite Schreier graphs for the action of
spinal groups on the levels of $T_d$ is well understood and can be
described recursively, see \cite{NP} and also \cite{MB} for the
special case of Grigorchuk's family. Translated in words in
$\mathcal A$, this recursion means that $\eta$ is the limit of words
of the shape
$$p^{(n+1)}=p^{(n)}\alpha_{\xi(n+1)}p^{(n)},$$ with $n>0$ and
$p^{(0)} = a$. Hence, by the definition of a simple Toeplitz subshift as given in Section~\ref{s:simple}, we conclude that our subshift is simple Toeplitz with eventual alphabet
$$\widetilde{\mathcal{A}}_{\xi} = \{ \alpha_\phi \in \mathcal A : \xi_j = \phi \ \text{ for infinitely many } j\in \mathbb N \} .$$
By definition of $\Xi_{2,m}$, for every $ \xi \in \Xi_{2,m}$, the intersection of kernels of the epimorphisms $\xi_{j}$ is trivial along any tail of $\xi$, hence the eventual alphabet consists of at least two letters for every $\xi\in \Xi_{2,m}$, $ m\geq 2$. The subshifts that we associate to the groups $ \{ G_{\xi} \ | \xi \in \Xi_{2,m} \}$, $ m\geq 2$, are therefore all aperiodic. We arrive at the following result.

\begin{theorem}\label{t:spinal-simple-Toeplitz}
The subshift over the alphabet $\mathcal A$ associated with a spinal
group $G_\xi$, $\xi\in \Xi_{2,m}$, is an aperiodic simple Toeplitz subshift.
\end{theorem}

This allows us to apply the Theorems~\ref{t:main-abstract} and \ref{thm-simple-toeplitz-lsc} to conclude Cantor spectrum result for anisotropic Markov operators. The anisotropicity of the Markov operator translates into a condition on the weights attached to the letters of the alphabet $\mathcal A$ that describe connections in
the Schreier graph. Namely, to the letter $ \alpha_{\phi} \in \mathcal A $ is attached the weight $ q_{\phi} = \sum_{ b \in B \setminus \Ker(\phi) } p_{b} $, which is positive as we have assumed $ p_{s} > 0 $ for all $ s \in S_{\xi} $. Here is the precise statement.

\begin{coro}\label{c:spectrum-spinal}
Let $G_\xi$ be a spinal group with $\xi\in \Xi_{2,m}$, $m\geq 2$. If
$M$ is a Markov operator on a graph $X\in Sch(G_\xi)$ such that the
numbers $q_\phi$ are not all equal over the essential alphabet $\widetilde{\mathcal{A}}_{\xi}$,
then the spectrum of $M$ is a Cantor set of Lebesgue measure $0$.
\end{coro}

\begin{proof}
By the previous theorem and the construction of the subshift, the
Markov operators in question can be considered as Jacobi operators
on a simple Toeplitz subshift. By
Theorem~\ref{thm-simple-toeplitz-lsc} simple Toeplitz subshifts
satisfy (LSC). In fact, the Markov operator on the Schreier graph
becomes exactly the Jacobi operator on the subshift with the
function $f$ taking values $p_{a}$ and $q_{\phi}$ over the essential
alphabet $\widetilde{\mathcal{A}}_{\xi}$. Proposition
\ref{prop:aperiodicity} then  ensures, in notations of
Theorem~\ref{t:main-abstract}, that $(f,g)$ is aperiodic and that
$f(\omega) \neq 0$ for all $\omega$, as required in
Theorem~\ref{t:main-abstract}.
\end{proof}

\begin{Remark}[Condition on the weights]
From the proof above we see that the conditions on the weights that
we really need are: $ p_{s} = p_{s^{-1}} $ for all $ s \in S_{\xi}$,
$p_{a} > 0$ as well as  $q_{\phi} > 0 $ and not all equal over the
essential alphabet $\widetilde{\mathcal{A}}_{\xi}$. The positivity
conditions ensure that the weighted graph is
connected\footnote{Should on the other hand $q_\phi=0$ hold for one
of the $\phi$ appearing in $\xi$, the corresponding graphs will be
an infinite union of finitely many finite graphs each. Hence, their
spectrum consists of finitely many eigenvalues (each with infinite
multiplicity).} and the last condition, as mentioned above, is the
anisotropicity of $M$ and is necessary for the aperiodicity of the
Jacobi operator on the corresponding subshift.
\end{Remark}

\begin{Remark}[Spectrum in the isotropic case]
The previous corollary deals with the spectrum of the anisotropic
Markov operator. Note that the spectrum of the isotropic Markov
operator can be computed explicitly. Namely, it is proven in
\cite{GNP} that if $G_\xi$ is a spinal group with $\xi\in
\Xi_{2,m}$, $m\geq 2$, then the spectrum of the isotropic Markov
operator on the orbital Schreier graph for the action on the
boundary of the tree is
\[ \left\lbrack - \frac{1}{2^{m-1}},0 \right\rbrack \cup \left\lbrack 1 - \frac{1}{2^{m-1}}, 1 \right\rbrack. \]
Moreover it coincides with the spectrum of the isotropic Markov
operator on the Cayley graph of this group with respect to the generating set $ S_{\xi} $.
\end{Remark}

\begin{Remark}[Spectrum of the Schreier graph vs spectrum of the Cayley graph]
As noted in the previous remark, for spinal groups $ G_{\xi} $, $
\xi \in \Xi_{2,m} $, the spectrum of the Markov operator on the
orbital Schreier graph for the action on the boundary of the
infinite binary tree coincides with that of the Markov operator on
the Cayley graph, with respect to the spinal generators $ S_{\xi} $.
This result is certainly not true in general, and even for these
groups it is now known if the spectra are the same in the
anisotropic case. Note however that under certain natural conditions
the spectrum of a Schreier graph embeds in the spectrum of the
Cayley graph (see \cite{DG}). More specifically, whenever the group
has intermediated growth part (b) of that theorem directly gives
inclusion of spectra. This inclusion holds also for amenable groups
acting on trees. Indeed, then the infinite Schreier graph $\Gamma$
is a limit of finite Schreier graphs $\Gamma_n$ (corresponding to
levels of the tree), and in the case this infinite graph  is
amenable then its spectrum is the closure of the union of spectra of
graphs $\Gamma_n$ (see \cite{BG}).
\end{Remark}

\end{document}